\documentclass[a4paper,12pt]{amsart}
\usepackage[utf8x]{inputenc}
\usepackage{amssymb,amsmath}

\newtheorem{claim}{Claim}
\newtheorem{lemma}{Lemma}
\newtheorem{theorem}{Theorem}
\newtheorem{cor}{Corollary}

\theoremstyle{definition}
\newtheorem{definition}{Definition}

\theoremstyle{remark}

\newcommand{\R}{\mathbb{R}}
\newcommand{\C}{\mathbb{C}}

\DeclareMathOperator{\sgn}{sgn}

\DeclareMathOperator{\dist}{dist}
\DeclareMathOperator{\supp}{supp}

\newcommand{\rhup}{\rightharpoonup}
\newcommand{\ba}{\mathbf{a}}

\newcommand{\bb}{\mathbf{b}}
\newcommand{\Z}{\mathbb{Z}}
\newcommand{\cH}{\mathcal{H}}

\newcommand{\La}{\Lambda}
\newcommand{\la}{\lambda}
\newcommand{\eps}{\epsilon}
\newcommand{\Mu}{\rm M}

\newcommand{\dr}{\downharpoonright}
\newcommand{\dl}{\downharpoonleft}

\renewcommand{\Im}{\mathop{\rm Im}\nolimits}
\renewcommand{\Re}{\mathop{\rm Re}\nolimits}

\newcommand{\B}{{B_\pi}}
\newcommand{\Ba}{{B_\pi^1}}

\newcommand{\li}{{\ell^\infty_1}}

\newcommand{\beq}{\begin{equation}}
\newcommand{\eeq}{\end{equation}}

\title[Bandlimited Lipschitz functions]{Bandlimited Lipschitz functions}

\author[Yu. Lyubarskii]{Yurii Lyubarskii}
\address{Dept. Mathematical Sciences, Norwegian University of Science
and Technology,NO-7491 Trondheim, Norway} \email{yura@math.ntnu.no}

\author[J. Ortega-Cerd\`a]{Joaquim Ortega-Cerd\`a}
\address{Dept.\ Matem\`atica Aplicada i An\`alisi,
Universitat  de Barcelona, Gran Via 585, 08007 Barcelona, Spain}
\email{jortega@ub.edu}


\thanks{The first author is partially supported by the Norwegian Research 
council, (projects  NOPIMA \# 185359 and DIMMA \#  213638). The second author is 
supported by the Generalitat de Catalunya (grant 2009 SGR 1303) and the Spanish 
Ministerio de Econom\'{\i}a y Competividad (project MTM2011-27932-C02-01). Part 
of this work was done while the authors were staying at the Center for Advanced 
Study, Norwegian Academy of Science, and they would like to express their 
gratitude to the institute for the hospitality}

\begin{document}

\begin{abstract}
 We study the space of bandlimited Lipschitz functions in one variable. In
particular we provide a geometrical description of  interpolating and
sampling sequences for this space. We also give a description of the trace of
such functions to sequences of critical density in terms of a cancellation
condition.
\end{abstract}

 \maketitle
\section{Introduction}

A standard model for one-dimensional bandlimited signals is the space  of 
functions (or distributions) that have Fourier transform $\hat f$ supported on a 
finite interval. According to the Paley-Wiener-Schwartz theorem the functions 
with compact  frequency support can be extended from the real line  into the 
whole complex plane $\C$   as   entire functions of exponential type. 

The size of a signal  is usually measured either in terms of its energy, i.e. 
the $L^2(\R)$ norm, or in the supremum $L^\infty(\R)$ norm. In the first case we 
encounter the familiar Paley-Wiener space of entire functions and in the second 
the Bernstein space (its definition is reminded on the next page).

One objection to the use of the Bernstein space as a model for bandlimited 
signals is that a very common operation in signal processing, the filtering, 
does not preserve the space. By filtering  we mean the operation that to $f$ 
corresponds a function $T(f)$ with Fourier transform $\hat f \chi_{w<0}$. Here 
$\chi_{w<0}$  denotes the characteristic  function of the negative axis. The 
content of the signal in all frequencies bigger than $0$ has been filtered out. 
The fact that the Bernstein space is not preserved by filtering is due to the 
unboundedness of the Hilbert transform in $L^\infty(\mathbb R)$: for $f\in 
L^\infty(\R)$ its Hilbert transform belongs to the space of functions   of 
bounded mean oscillation $BMO(\R)$.  We refer the reader to  
\cite[Ch.~6]{Garnett} for the definition and basic properties of functions in 
this space.  In view of this, a natural substitute for the Bernstein space has 
been proposed in \cite{Thakur11}. It consists of entire functions of exponential 
type that, when restricted to $\R$, belong to $BMO$.

It was observed, see \cite[Thm~7]{Thakur11}, that the bandlimited functions in 
$BMO$ enjoy a much better regularity than expected, they have bounded derivative 
on $\mathbb R$, i.e., they are Lipschitz. We provide a short argument showing 
this: from the Fourier transform side taking derivative and applying the Hilbert 
transform corresponds to multiplying the function by $-i\omega$ and by 
$\sgn(\omega)$ respectively. If the function $f$ has Fourier transform supported 
on $[-\pi,\pi]$, say,  then $f'=f\star \phi$, where $\hat\phi$ is any compactly 
supported smooth function which coincides with $i \omega$ on $[-\pi,\pi]$. If 
$f$ is a $BMO$ function then $f=Hg+h$, where $g,h\in L^\infty$ and    $H$ is the 
Hilbert transform.  If, in addition, $\supp \hat f \subset [-\pi,\pi]$, then 
$f'= \psi\star g+ \phi\star h$, where $\phi$ is as above and $\hat \psi(\omega) 
=\sgn(\omega) \hat \phi(w)$. Both $\phi$ and $\psi$ belong to  $L^1(\R)$ and 
$f'$ is therefore bounded.

Thus it seems natural to consider  the following spaces: The  Bernstein space of 
entire functions:
\[
\B= \{F\in L^\infty(\R) ,  \supp \hat F\subset[-\pi,\pi]\}
\]
endowed with its natural norm $\|F\|_\B:=\sup_{\R}|F(x)|$ and the space
\[
\Ba=\{ F,\ F' \in \B  \}, 
\]
endowed with its natural seminorm: $\|F\|_{\Ba}:=\sup_{\R}|F'(x)|$. Clearly by 
the Bernstein inequality $\B\subset \Ba$ but the converse is not true, the 
function $f(x)=x$ belongs to $\Ba\setminus \B$. 

A fundamental problem in the study of bandlimited functions is the process of
discretization of signals. This problem can be decoupled in two:

\begin{itemize}
\item  The problem of stable reconstruction of  a signal from the set of its 
samples at a given sequence of points $\La \subset \R$. If this is possible we 
say that $\La$ is a sampling sequence. 
\item Its companion problem of prescribing an arbitrary set of values on a 
sequence $\La \subset \R$. If this is possible  we say that  $\La$ is an 
interpolating sequence.
\end{itemize}

Beurling in \cite[Chapt. IV, V]{Beurling89a} considered  both problems in the Bernstein space 
and provided a complete geometric description of sampling and interpolating 
sequences. We aim to do such description for the space $\Ba$.  The major 
difference between the two settings is related to the fact that the  set of 
traces of functions in $\Ba$ is now defined by the divided differences of the 
values rather than the values themselves, so the (now) classical  machinery 
from \cite{Beurling89a} cannot be applied directly. We need to combine this 
machinery with  additional tools  in order to deal with spaces defined through 
the derivatives. In particular we use ideas from   \cite{BoeNic04}, this 
article deals with the Bloch space and also  \cite{LM05}, this article studies  
functions whose derivatives are in the Paley-Wiener space. 
 

We now introduce the corresponding spaces of sequences.

Given  $\Lambda=\{\lambda_k\}_k\subset \R$, $\lambda_k<\lambda_{k+1}$,   and a
function $a:\Lambda \to \C$ we denote the divided differences:
\[
  \Delta_a(\lambda_k)= \frac
{a(\lambda_{k+1})-a(\lambda_k)}{\lambda_{k+1}-\lambda_k},
  \]
and consider the space
\begin{align*}
\ell^\infty_1(\Lambda)= \{a:\Lambda \to \C; \| a
\|_{\ell^\infty_1(\Lambda)}:=\sup_{k} |\Delta_a(\lambda_k)| < \infty\},  
\end{align*}

Equivalently  it can be defined as
\[
\ell^\infty_1(\Lambda)= \{a:\Lambda \to \C; \| a
\|_{\ell^\infty_1(\Lambda)}:=\sup_{\la,\la' \in \La, \  \la\neq \la'} \left \{ 
\left | \frac{a(\la)-a(\la')}{\la -\la'} \right |  \right  \} < \infty\}.
\]

It is clear that, given $F\in \Ba$ we have that $a\in \li$ if 
$a(\la_k)=F(\la_k)$.

\begin{definition}
We say that a sequence $\Lambda$ is \emph{sampling} for $\Ba$  if 
\[
  \|F\|_{\Ba} \leq C \|\{F(\lambda_k)\}\|_\li, \ F\in \Ba 
\]
with some $ C<\infty$ independent of the choice of $F$.
\end{definition}
By $K=K(\Lambda)$ we denote the \emph{sampling constant}, i.e. the smallest
possible value of $C$ in the above inequality. 

\begin{definition} 
We say that  a sequence $\Lambda$ is \emph{interpolating} for $\Ba$  if for
each $a\in \li(\Lambda)$  there is an $ F\in \Ba$ such that 
 \begin{equation}
\label{interpolation}
 F(\lambda_k)=a(\lambda_k),  \ k\in \Z.
 \end{equation}
 \end{definition}
If a sequence $\Lambda$ is interpolating, then by the closed graph theorem it is 
possible to interpolate with a size control. That is, there is a constant $C$ 
such that, for each $a\in \li(\La)$, one can choose   $F$ interpolating $a$ on 
$\Lambda$ and  in addition $\|F\|_{B_\pi^1}\le C 
\|a\|_{\ell_1^\infty(\Lambda)}$. The smallest constant possible in this 
inequality is called the \emph{interpolation constant} and it will be denoted by 
$K_0(\Lambda)$.

The problem of describing of sampling and interpolating sequences sequences in 
$\Ba$   presents an interesting challenge because the routine interpolation 
tools such as Lagrange interpolation series cannot be applied directly to 
interpolation by divided differences. We develop technique of interpolation 
related to $\bar{\partial}$ problem and combine them with the classical 
techniques in \cite{Beurling89a}.
 
Our first aim is to provide a complete geometric characterization of
interpolating and sampling sequences in $\Ba$. We introduce now the geometric
concepts that are used in the description.

We say that a sequence $\Lambda\subset \mathbb R$ is \emph{separated}
whenever
\[
\alpha := \inf \{|\la-\la'|,  \ \la,\la'\in \La, \ \la\neq \la'\}>0,
\]
we say that $\alpha$ is the \emph{separation constant}  for $\La$.

We will also use the classical notions of upper and lower Beurling densities:
\[
D^+(\Lambda)= \lim_{R\to \infty} \sup_{x\in \R} \frac{\#(\Lambda\cap (x,x+R))}
R,
\]
\[
D^-(\Lambda)= \lim_{R\to \infty} \inf_{x\in \R} \frac{\#(\Lambda\cap (x,x+R))}
R.
\]

\begin{theorem}
\label{thm:inter}
A sequence $\Lambda \subset \R$ is interpolating for $\Ba$ if and only if
$\Lambda=\Lambda_1\cup \Lambda_2$ where $\Lambda_1$, $\Lambda_2$  are
two separated
sequences, $\Lambda_1\cap \Lambda_2=\emptyset$ and  $D^+(\Lambda) < 1$.
\end{theorem}

\begin{theorem}
\label{thm:sampl}
A sequence $\Lambda \subset \R$ is sampling for $\Ba$ if and only if there are
separated  sequences  $\Lambda_1,\Lambda_2\subset \Lambda$, $\Lambda_1\cap
\Lambda_2=\emptyset$ and  $D^-(\Lambda_1\cup \Lambda_2) > 1$.
\end{theorem}
{\bf Remark.}
 The fact that that $\Lambda_1\cap \Lambda_2=\emptyset$ is not relevant. This
is due to the fact that for the sake of simplicity in the exposition we are not
considering points with multiplicity. If we did, we would have to introduce
derivatives at the multiple points replacing divided differences.

We will also study two further problems in these spaces. As one can see from 
Theorem~\ref{thm:inter}, the trace of a function $F\in \Ba$ on a sequence $\Lambda$ of 
density one is not arbitrary sequence with bounded divided differences. Apart 
from this necessary   condition it must also satisfy a certain 
cancellation property. 

This was studied by Levin in \cite[Appendix VI]{Levin56}, see also 
\cite{Levin96}, in the case of the Bernstein space and $\La$ being the set of 
integers. We carry out the analogous result in the context of bandlimited 
Lipschitz functions. The characterization of the traces  on a sequence which is 
a zero set of a sine-type function (the integers is the main example) is 
achieved through the use of the discrete and regularized Hilbert transform, in 
a similar spirit as   for the Bernstein space and integer nodes, yet 
an additional regularization is needed. One consequence of our result is that 
it is possible to reconstruct the function from its values in the zeros of a 
sine-type function plus the value in any other given point. Of course the 
reconstruction is not stable in view of Theorem~\ref{thm:sampl}. This is the 
case also in the Bernstein space but it is curious that this is possible   in 
the strictly bigger  space $\Ba$. This fact has already been observed in 
\cite{Thakur11}.

{\bf Remark} It is interesting to know how wide a space $X$ of functions of 
exponential type $\pi$ can be such that the zero set of a sine-type function 
plus one point are sets of uniqueness for $X$. We do not know the general 
answer to this question. Yet we observe that this property is more related to 
the regularity of the Fourier transform (in the distributional sense)  of 
functions in $X$ near the endpoints of the spectra, rather than their growth 
properties. In particular the spaces considered in \cite{LM05} possess this 
property yet contain functions of polynomial growth.

The structure of the paper is as follows. In Section~\ref{sect:sampling} we 
prove Theorem~\ref{thm:sampl} providing a description of sampling sequences in 
$\Ba$. In Section~\ref{sect:interpolating1} we prove the necessity part of the 
interpolating   Theorem~\ref{thm:inter} and in 
Section~\ref{sect:interpolating2} we prove the sufficiency of the geometric 
description. In Section~\ref{sect:traces} we provide a description of the traces 
of functions in $\Ba$ on the zero sets of sine-type functions. 

We will use the following notation: given two positive quantities $a$ and $b$ we 
write $a \lesssim b $ or $b \gtrsim a$ if there is a constant $C>0$ such that 
$a\leq Cb$ for all possible values of parameters. We write $a\simeq b$ if $a 
\lesssim b$ and $b \lesssim a$. 


\section{Sampling sequences}\label{sect:sampling}
The strategy for the sampling part is to reduce the problem to the analogous
problem in $\B$.
\begin{definition} 
A  sequence  $\Lambda\subset \R$ is called \emph{relatively dense} if there is $R>0$ 
such that $\La\cap (\xi-R,\xi+R) \neq \emptyset$ for each $\xi\in \R$.
\end{definition}

\begin{claim}\label{claim:reldense}
Let $\Lambda\subset \R$ be a sampling sequence  for $\Ba$.  Then $\Lambda$ is 
relatively dense.
\end{claim}
\begin{proof}
Let the opposite be true: for any $R>0$ there is a $\xi\in\R$ such that 
$\La\cap (\xi-R,\xi+R)= \emptyset$. Consider the function 
\[
F_\xi(z):=\int_\xi^z \frac{\sin(\zeta-\xi)}{\zeta-\xi}d\zeta \ \in \B.
\]
We clearly have $\|F_\xi\|_{\Ba}=1$. On the other hand $\|F|_\La\|_{l^\infty_1}\lesssim 
R^{-1}$ since $|F'_\xi(x)|\lesssim R^{-1}$ as $x\not\in (\xi-R,\xi+R)$ and 
also $\|F_\xi\|_{\B}\lesssim 1$. This contradicts the sampling property of 
$\La$.
\end{proof}

\begin{claim}
If  a separated sequence $\Lambda$ is sampling for $\Ba$ then 
$D^{-}(\Lambda)>1$.
\end{claim}
\begin{proof}
We will prove that in this case $\Lambda$ is sampling for $\B$  and we may use
then the results by Beurling, \cite{Beurling89a}. 

It suffices to prove that, for any function   $F\in \B$ the inequality 
$\|F|_\La\|_{l^\infty} <1$ yields $\|F\|_{\B}<C$ with some constant $C$ 
independent of the choice of $F$. Indeed since $\La$ is separated then  
$\|F|_\La\|_{l^\infty} <1$ yields $\|F|_\La\|_{\li} \lesssim1$ and $F'\in \Ba$, 
$\|F'\|_{\B} \lesssim 1$ because $\Lambda$ is  sampling for $\Ba$. The desired 
inequality follows now from the fact that, being a sampling sequence for $\Ba$, 
the sequence $\La$ is relatively dense.
\end{proof}

Given two sequences $\Lambda, \Gamma\subset \R$ we say that 
$\dist_H(\Lambda,\Gamma) < \eps$  if $\# \{[\mu-\eps, \mu+\eps]\cap 
\Lambda\}\geq 1$ for each $\mu\in \Gamma$ and  $\# \{[\lambda-\eps, 
\lambda+\eps]\cap \Gamma\}\geq 1$ for each $\lambda\in \Lambda$. Here $\dist_H$ 
stands for the Hausdorff distance.

We say that the sequences  $\La_k$  \emph{converge weakly to $\La$} if, for each 
$N>0$
\[
 \dist_H((\La\cap[-N,N])\cup \{-N,N\}, (\La_k\cap [-N,N])\cup \{-N,N\}) \to 0, 
\ \text{as} \ k \to \infty.  
\]
In this case we write $\La_k \rightharpoonup \La$.

Given a sequence $\La$ we denote by $W(\La)$ the set of all its weak limits of 
translates, i.e. sequences $\Mu$ such that 
\[
 \La+x_k \rightharpoonup \Mu, \ \text{for some \ $ \{x_k\}\subset \R$.}
\] 
  
We will use the two following stability results. Their proof mimics the  one  
of \cite[Chapter iV, Thm.~2]{Beurling89a} with the natural modifications: the divided 
differences should be approximated by  the first derivatives  and then 
Bernstein's theorem for the \emph{second} derivative should be applied.

\begin{claim}[First stability result]
Let $\eps>0$ and let the sequences $\La$ and $\Mu$ be such that each $\mu \in 
\Mu$ has at least two $\eps$-neighbors in $\La$, i.e. 
$\#\bigl\{(\Lambda\setminus \Mu)\cap [\mu-\varepsilon,\mu+\varepsilon]\bigr\}\ge 
2$ for all $\mu\in \Mu$.  Then
\[
|K(\Lambda) ^{-1}-K(\Lambda\cup \Mu)^{-1}| < 10 \eps
\]
\end{claim}
 
\begin{cor}
If $\Lambda$ is sampling then there is an $\eps>0$ and a subsequence 
$\Sigma\subset \Lambda$ which  is also sampling and such that $\sharp\Sigma\cap 
(x,x+\varepsilon)\le 2$  for each $x\in \R$. In other words  each sampling 
sequence $\La$ contains a sampling sequence $\Sigma$ which is a union of two 
separated sequences.   
\end{cor}

\begin{claim}[Second stability result] 
Let $\Gamma_1, \Gamma_2 \subset \R$ be two separated sequences 
$\Gamma_1\cap\Gamma_2=\emptyset$  and $\Sigma=\Gamma_1\cup \Gamma_2$.  
Let also $\Gamma_2'$ be a separated sequence $\Gamma_1\cap\Gamma_2'=\emptyset$, 
$\Sigma'=\Gamma_1\cup \Gamma_2'$, and $\dist_H(\Gamma_2,\Gamma_2') < \eps$. 
Then 
\[
\left | \frac 1 {K(\Sigma)} - \frac 1 {K(\Sigma')} \right | \leq 10 \varepsilon.
\]
\end{claim}

\begin{cor}
If $\Lambda$ is a sampling sequence which is a union of two separated sequences, 
then, for some $\eps>0$  there is a separated sampling sequence    $\Lambda'$  
such that $\dist_H(\La,\La')<\eps$.
 \end{cor}
 
\begin{theorem}
$\Lambda$ is a sampling sequence for $\Ba$ if and only if  it contains a 
subsequence $\Sigma$ which is the union of two separated sequences with 
$D^{-}(\Sigma)>1$.
\end{theorem}

\begin{proof}
The \emph{necessity} part is just a compilation of the previous claims. 

Now let $D^{-}(\Sigma)>1$ and $\Sigma$ be a union  of two  separated sequences.
We follow the arguments in \cite[Chapter IV, theorem 3]{Beurling89a}: it suffices to prove that each 
$\Sigma'\in W(\Sigma)$ is a uniqueness set for $\Ba$. 

Take any $\Sigma'\in W(\Sigma)$. We still have $D^-(\Sigma)>1$, here one has to 
count points according to their multiplicities, some points in a weak limit may 
have  multiplicity two. The corresponding divided difference should be replaced 
by the derivative then. It suffices to prove uniqueness  for functions $F\in 
\Ba$ such that $F(x)\in \R$ for all $x\in \R$. Each such function has zero 
increments on $\Sigma'$ then its derivative is vanishing at some intermediate 
points, the set of intermediate points has density bigger than one so $F'=0$.
 \end{proof}

 \begin{cor} If $\Lambda$ is sampling for $B_\pi^1$ then there is
an $\varepsilon>0$ such that $(1+\varepsilon)\Lambda$ is still sampling for
$B_\pi^1$.
 \end{cor}


\section{Interpolation theorem. Necessity.}\label{sect:interpolating1}

Let $\Lambda$ be an interpolating sequence for $\Ba$. First we prove that 
$\Lambda=\Lambda_1\cup \Lambda_2 $  where $\Lambda_1,\Lambda_2$ are separated
sequences, $\Lambda_1\cap \Lambda_2=\emptyset$.  This follows from the simple
statement below.

\begin{claim}  There is $\eps >0$ such that, for each $\lambda\in \Lambda$ 
\[
\#(\Lambda\cap \{\zeta; |\zeta-\lambda|<\eps|\})\leq 2.
\] 
\end{claim}

\begin{proof} 
Indeed, if this were not the case then for each $\eps >0$ there would exist 
points $\lambda_{k-1},\lambda_k,\lambda_{k+1}\in \Lambda$, 
$\lambda_{k+1}-\lambda_{k-1}< \eps$. One can construct $a\in \li$, $\|a\|_\li=1$ 
so that $a(\lambda_k)=0$, $a(\lambda_{k-1})= \lambda_k-\lambda_{k-1}$, 
$a(\lambda_{k+1})=\lambda_{k+1}-\lambda_k$. Let $F\in \Ba$ be a solution of 
\eqref{interpolation}  satisfying $\|F\|_\Ba \leq K_0(\Lambda)$ and, hence, 
$\|F''\|_\B \leq \pi K_0(\Lambda)$. On the other hand the choice of 
interpolation data yields $\|F''\|_\B \geq \eps^{-1}$, so $\eps $ cannot be 
chosen arbitrary small.
\end{proof}
We can now split the sequence $\Lambda$ into blocks  $\Lambda=\cup_j
\Lambda^{(j)}$   containing    at most two points each: either $\Lambda^{(j)}=
\{\lambda_{k_j}\}$ or $\Lambda^{(j)}= \{\lambda_{k_j}, \lambda_{k_j+1}\}$, in
addition $k_j<k_{j+1}$ and $\mbox{dist}(\Lambda^{(j)}, \Lambda^{(j+1)}) > a>0$
for all $j\in \Z$.
 
 Denote $\Gamma^{(j)}=\Lambda^{(j)}$  if $\#\Lambda^{(j)}=1$ and
$\Gamma^{(j)}=\{\lambda_{k_j}, \lambda_{k_j+1}+\eps\}$ otherwise,
the number $\eps$ will be chosen later. Let
 $\Gamma=\cup_j \Gamma^{(j)} =\{\mu_k\}$ with  enumeration 
corresponding to those of $\La$.  
 
 \begin{claim}
 If  $\eps$ is sufficiently small then $\Gamma$ is an interpolating sequence 
for
$\Ba$.
\end{claim}

\noindent{\em Proof}. 
Let for definiteness $\la_0\in \La\cap \Gamma$.  Given any $a\in
\li(\Gamma)$
we may assume $a(\la_0)=0$.  We use induction to construct the sequences
$a^{(p)}\in \li (\Gamma)$,   $c^{(p)}\in \li (\Lambda)$  $p=0,1,\ldots $.

Set $a^{(0)}=a$. Given $a^{(p)}$ construct $c^{(p)}\in \li(\La)$ such that
$c^{(p)}(\la_0)=0$,
\[
\frac{a^{(p)}(\mu_{k+1})-a^{(p)}(\mu_k)}{\mu_{k+1}-\mu_k}=
   \frac{c^{(p)}(\lambda_{k+1})-c^{(p)}(\lambda_k)}{\lambda_{k+1}-\lambda_k}
\] 
and let $F_p\in \Ba$ solve the interpolation problem
 $F_p(\la_k)=c^{(p)}(\la_k)$ and $\|F_p\|_\Ba \leq K_0(\Lambda)
\|a^{(p)}\|_{\li(\Gamma)}$.
 Further let 
 $$
 a^{(p+1)}(\mu)=a^{(p)}(\mu)-F_p(\mu), \ \mu \in \Gamma.
 $$
 
 We claim that for sufficiently small $\eps$ there is $q\in (0,1)$, such that
$qK_0(\Lambda)<1$ and
 \begin{equation}
 \label{contraction}
 \|a^{(p+1)}\|_{\li(\Gamma)} \leq q \|a^{(p)}\|_{\li(\Gamma)}, \
\|F_{p+1}\|_\Ba  <
qK_0(\Lambda) \|F_p\|_\Ba.
 \end{equation} 
  
If this is proved we will use that $F_p(\la_0)=0$ for all $p$'s, hence the 
series $F=\sum F_p$ converges on each compact set in $\C$ and delivers a solution to 
  the interpolation problem
  \[
  F(\mu)=a(\mu), \ \mu \in \Gamma, \ F\in \Ba.
  \]
  
It remains to prove \eqref{contraction}.  Without loss of generality we may
assume that all data $a^{(p)}$ are real  and also the functions $F_p$ are real
on $\R$.  Let $\mu_k=\la_k$, $\mu_{k+1}=\la_{k+1}+\eps$ We have

\begin{equation}
\label{eq2:1}
\begin{aligned}
&\frac{a^{(p+1)}(\mu_{k+1})-a^{(p+1)}(\mu_k)}{\mu_{k+1}-\mu_k} =\\
&=\frac{a^{(p)}(\mu_{k+1})-a^{(p)}(\mu_k)}{\mu_{k+1}-\mu_k}-
    \frac{F_p(\mu_{k+1})-F_p(\mu_k)}{\mu_{k+1}-\mu_k}= \\
& =\frac{c^{(p)}(\la_{k+1})-c^{(p)}(\la_k)}{\la_{k+1}-\la_k} -  
\frac{F_p(\mu_{k+1})-F_p(\mu_k)}{\mu_{k+1}-\mu_k}=\\
&=\frac{F_p(\la_{k+1})-F_p(\la_k)}{\la_{k+1}-\la_k}  -
\frac{F_p(\mu_{k+1})-F_p(\mu_k)}{\mu_{k+1}-\mu_k}.
\end{aligned}
\end{equation}
Let  $\la_{k+1}-\la_k< 1/10\pi$. Then  
\[
 \frac{a^{(p+1)}(\mu_{k+1})-a^{(p+1)}(\mu_k)}{\mu_{k+1}-\mu_k}=
F_p'(\tilde{\la}_k)-F_p'(\tilde{\mu}_k),
\]
for some $\tilde{\la}_k\in (\la_k,\la_{k+1})$, $\tilde{\mu}_k\in (\lambda_k,
\la_{k+1}+\eps)$. For $\eps< 1/10\pi$ we obtain
\[
|\tilde{\mu}_k-\tilde{\la}_k| \leq \frac 1 {5\pi} \ \mbox{and} \ \left |
\frac{a^{(p+1)}(\mu_{k+1})-a^{(p+1)}(\mu_k)}{\mu_{k+1}-\mu_k} \right | < \frac 1
{5}  \|F_p\|_\Ba.
\]

In the case  $\la_{k+1}-\la_k\geq 1/10\pi$ we have 
\begin{multline*}
 \frac{a^{(p+1)}(\mu_{k+1})-a^{(p+1)}(\mu_k)}{\mu_{k+1}-\mu_k} = \\
 \frac{F_p(\la_{k+1})-F_p(\la_k)}{\mu_{k+1}-\mu_k}\frac{\mu_{k+1}-
\la_{k+1}}{\la_{k+1}-\la_k} -
           \frac{F_p(\la_{k+1})- F_p(\mu_{k+1})}{\mu_k-\la_k}
\end{multline*}
and an explicit estimate shows
\[
\left | \frac{a^{(p+1)}(\mu_{k+1})-a^{(p+1)}(\mu_k)}{\mu_{k+1}-\mu_k} \right |
\lesssim \eps \|F\|_\Ba. 
\]
 Relation \eqref{contraction}  now follows.   \qed

\begin{cor} Without loss of generality we can assume that the
sequence $\Lambda$ is separated.
\end{cor}



We will prove that if $\La$ is an interpolating set for $\Ba$, then one can 
refer to the classical Beurling result in order to get $D^+(\Lambda)<1$.
 
It suffices to construct a sequence of functions $\{f_\la\}_{\la \in \La}
\subset \B$  such that
\begin{equation}
\label{eq:delta}
f_\la(\mu)=\delta_{\la, \mu}, \ \mu \in \La,
\end{equation}
and
\begin{equation}
\label{eq:estimate}
|f_\la(x)|\lesssim \frac {1}{|x-\la|^2+1}, \ x \in \R.
\end{equation}

Then $\La$ will be an interpolating sequence for $\B$ since the solution to the
interpolation problem 
\[
F(\la)=b(\la); \ b \in \ell^\infty(\La), \ F \in \B,
\]
can be achieved by the function
\[
F_b= \sum b(\la) f_\la, 
\]
and according to \cite[Chaprer V, Theorem 1]{Beurling89a}  $D^+(\La)<1$.

\medskip 

The construction of the functions  $\{f_\la\}_{\la \in \La} $ relies on the
following statement:

\begin{claim}
\label{extrapoint}
Let $\Lambda$ be a separated interpolating sequence for $\Ba$.
For each $\xi\in \R\setminus \La$ the sequence $\La\cup \{\xi\}$ is also
interpolating for $\Ba$. Moreover the constant
$K_0(\La\cup \{\xi\})$ depends only on $\dist(\La, \xi)$.
\end{claim}

\medskip

Let us take this claim for granted  for the moment being. Let $\alpha$ be the 
separation constant  for $\La$. For each  $\la\in\La$, choose the points $ 
\xi_\pm=\la\pm \alpha/8$, $\xi_1=\la+\alpha/4$.  The set $\La_\la:=\La\cup 
\{\xi_+,\xi_-, \xi_1\}$ is $\Ba$-interpolating and  also $K_0(\La_\la)<C$, $C$ 
being independent of the choice of $\la\in\La$. 

Define the sequence  $a_\la\in \ell^\infty_1(\La_\la)$ as 
\begin{equation}
\label{eq:ala}
a_\la(\mu)= \begin{cases} 1, & \mu = \la;\\  0 , & \mbox {otherwise}            
          \end{cases}
 \end{equation}
and take $g_\la\in \Ba$ such that $g_\la(\mu)= a(\mu), \ \mu \in \La_\la$, 
$\|g_\la\|_{\Ba}\lesssim 1$. It is straightforward that one can find numbers 
$c_\la$ such that $|c_\la| \sim 1$ and the  functions \[ f_\la(z) = c_\lambda 
\frac {g_\la(z)}{(z-\xi_-)(z-\xi_+)(z-\xi_1)} \] satisfy   \eqref{eq:delta} and 
\eqref{eq:estimate}. 

\medskip

In order to verify Claim \ref{extrapoint} it suffices to prove that, for each 
$\xi \in \R\setminus \La$, there is a function  $h_\xi \in \Ba$ such that 
\[
h_\xi|_\La=0, \ h_\xi(\xi)=1.
\]
 and $ \|h_\xi\|_\Ba $  can be estimated by a quantity which only depends on 
$\dist(\La, \xi)$. 

We   mimic the proof of the corresponding fact in \cite{Beurling89a}, Chapter V.
 
 \begin{claim}
 \label{beurling_limits}
If $\La$ is a separated  interpolating sequence for $\Ba$ then each $\Gamma\in 
W(\La)$ is also an interpolating sequence for $\Ba$, in addition 
$K_0(\Gamma)\leq K_0(\La)$.
 \end{claim}

The proof follows that in \cite[Chapter V, Lemma~5.]{Beurling89a}.

\medskip

Given $\xi\in \R\setminus \La$ denote 
\[
\rho_\La(\xi) = \sup \{|F(\xi)|, F\in \Ba,  \ F|_\La=0, \ \|F\|_\Ba \leq 1 \}.
\]

\begin{claim}
Let $\La$ be an $\alpha$-separated interpolating sequence. Then
for each $\alpha'\in (0, \alpha/2)$ there is   $\kappa>0$ such that 
\begin{equation}
\label{rhoestimate}
\rho_\La(\xi) >\kappa,    \text{ if }    \dist(\xi, \La)>\alpha'
\end{equation} 
\end{claim}
\begin{proof} 
First we mention that    $\rho_{\La}(\xi)>0$, $\xi\notin \La$. Indeed otherwise
$F\in \Ba$, $F|_\La=0$ yields $F=0$ in other words the mapping $T:F\mapsto
F|_\La$, $T:\Ba \to \ell^\infty_1(\La)$ has zero kernel. Since $\La$ is an
interpolating sequence $T$ acts onto and hence is invertible. This means that
$\La$ is also a sampling set and, by Theorem 3, $D^-(\La)>1$. Take any three
points $\la_1, \la_2, \la_3 \in \La$ and denoting $\La'=\La\setminus
\{\la_i\}_{i=1,2,3}$ we have $D^-(\La')>1$, hence $\La'$ is a sampling set as
well, this contradicts the fact that $\La$ is an interpolating sequence.

It follows now that, if $\La$ is an interpolating sequence and $\xi \not \in 
\La$, then $\La\cup \{\xi\}$ is also an interpolating sequence.
 
Assume that  there
is a sequence of points $\xi_n$, $\dist(\xi_n, \La)>\alpha'$ and
$\rho_{\La}(\xi_n) \to 0$ as $n\to \infty$. Let $\La_n= \La-\xi_n$.  Each
$\La_n$ is an $\alpha$-separated sequence and also  $K_0(\La_n)=K_0(\La)$,
$\rho_{\La_n}(0)\to 0$. We may assume that $\La_n\rhup \Gamma$, then
$\dist(0,\Gamma) \geq\alpha'$.  Fix two points $t_1,t_2 \not\in \Gamma\cup \{0\}$.
The set $\Gamma'=\Gamma\cup \{t_1,t_2\}$ is also an interpolating sequence, 
hence $\gamma:=\rho_{\Gamma'}(0)>0$.

Therefore there exists $F\in \Ba$ such that  $\|F\|_\Ba=1$,  $F|_{\Gamma'}=0$,
and $F(0)=\gamma$. Then 
\[
G(z):=\frac{F(z)}{(z-t_1)(z-t_2)} \in \Ba, \  G|_{\Gamma}=0, \ G(0)=\gamma
t_1^{-1}t_2^{-1}, 
\]
in addition  $G'(x) \to 0$, as  $x\to \infty$.

We have   now   $\|G|_{\La_n} \| _{l^\infty_1(\La_n)} \to 0$ as $n\to
\infty$: for large values of the argument this follows from the decay of $G'$
(we remind that all $\La_n$  are $\alpha$-separated) for limited values of the
argument this follows from the fact that $G|_{\Gamma}=0$ and $\La_n\rhup
\Gamma$. Since $K_0(\La_n)=K_0(\La)$ we can find a function $H_n\in \Ba$ such
that
$H_n|_{\La_n}= G|_{\La_n}$ and    also $\|H_n\|_{\Ba}\to 0$ as $n\to \infty$. In
addition  $H_n(0)\to 0$, $n\to \infty$.  Now the functions 
\[
\Phi_n(z)=G(z)-H_n(z)
\]
satisfy $\Phi_n|_{\La_n}=0$, $\|\Phi_n\|_{\Ba}<C$ and also $\Phi_n(0)\to \gamma
t_1^{-1}t_2^{-1}$. The latter is incompatible with $\rho_{\La_n}(0)\to 0$.  
\end{proof}
 
\section{Interpolation theorem. Sufficiency.}\label{sect:interpolating2}
Let  $\La=\La_1\cup \La_2\subset \R $  satisfy the hypothesis of
Theorem~\ref{thm:inter}, i.e. the subsequences $\La_1,\La_2$ are separated,
$\La_1\cap\La_2= \emptyset$,  and $D^+(\La)<1$.  We are going to prove that
$\La$ is interpolating for $\Ba$.

\begin{claim}
\label{cl: extra point}
Let a sequence $\Sigma\subset \R$, $\Sigma\cap \La= \emptyset$  be such that
 $\Gamma=\La\cup \Sigma$ is interpolating for $\Ba$. Then $\La$ is also
interpolating for  $\Ba$ with the same constant of interpolation. 
\end{claim}
\begin{proof}  
 Given any data $a\in
\li(\La)$ we may extend it on $\Lambda\cup \Sigma$ with the same Lipschitz
constant as
\[
\widetilde a(x)=\inf_{\lambda\in \La} (a(\lambda)+\mathop{Lip}(a) 
|x-\lambda|),\ \forall x\in \Lambda\cup \Sigma
\]
where $\mathop{Lip}(a)$ is the Lipschitz constant of $a$.

We   have then $a\in \li(\Gamma)$ and  $\|a\|_{\li(\Gamma)}= 
\|a\|_{\li(\La)}$. Any
solution of the corresponding problem on $\Gamma$ gives now a solution on $\La$.
\end{proof}

\begin{cor}
Without loss of generality we may assume $D^-(\La)>0$.
\end{cor}
Indeed, were this not the case one can add a relatively dense sequence of points
$\Sigma$ such that the union $\Gamma=\Lambda\cup \Sigma$ still has the property
$D^-(\Gamma)<1$ and we apply the previous Claim~\ref{cl: extra point}.

We start with proving the sufficiency assuming in addition that $\La$ is a 
separated sequence.  The corresponding machinery is related to the notion of 
sine-type functions.
\begin{definition}
 An entire function $S$ is a sine-type function if it is of exponential type, 
its zeros are simple and separated,  and there is a constant $C$ such that 
\begin{equation*}
 |S(z)|\simeq e^{\pi |\Im z|}, \qquad \forall z,\ |\Im z|>C.
\end{equation*}
\end{definition}
This definition was introduced by Levin, see e.g. \cite{Levin96} who proved that 
the zero set of a sine-type function, which lies in a strip around the real 
axis, is both an interpolating   and a sampling sequence for the Paley-Wiener 
space.

\begin{lemma}  
\label{lemma1}
Let $\La=\{\la_k\}\subset \R$ be separated, $D^+(\La)<1$ and $D^-(\La)>0$. Then
$\La$ is interpolating for $\Ba$. 
\end{lemma}

\begin{proof}
Since the upper density of the sequence $\Lambda$ is strictly smaller than one, 
it is possible to find a sequence $\Sigma \subset \R$ and a sine-type function 
$S$  with zero set $ \Lambda \cup \Sigma$. This is done in 
\cite[Lemma~3]{OrtSei98}.

Moreover for for each $\lambda\in \R$   one can consider the sequences 
$\Lambda_\lambda=\Lambda \cup \{\lambda +i\}$ and, with the same proof as in 
\cite{OrtSei98}, one can find sequences $\Sigma(\lambda)$ and functions 
$S_\lambda$ with zero sets $\Lambda \cup \{\lambda  +i\}\cup \Sigma(\lambda)$ 
and such that
\begin{equation}
 \label{sine-type}
 |S_\lambda(z)|\simeq e^{\pi |\Im z|}, \qquad \forall z,\ |\Im z|>C,
\end{equation}
with the constants implicit in \eqref{sine-type} being  uniform for all 
$\lambda\in \Lambda$. We split the remaining construction into several steps

\medskip

\noindent {\it Elementary solutions:}  
The hypothesis imply that there is an $\alpha>0$  such that
$\la_{k+1}-\la_k>\alpha>0$.

Denote $\gamma_k= \la_k+\alpha/10+i\R$
and let 
\[
\chi_k(z)=\begin{cases}
  0, & \text{if $\Re z < \la_k+\alpha/10$;} \\
  1, & \text{ if $\Re z >  \la_k+\alpha/10$.}
                \end{cases}
\]

Consider the functions
\begin{equation}
\label{blocks}
\Phi_k(z)=\frac{1}{2i\pi} \ \frac{S_{\la_k}(z)}{z-(\la_k+i)} \int_{\gamma_k}
           \frac{\zeta-(\la_k+i)}{S_{\la_k}(\zeta)}\frac{d\zeta}{\zeta-z} +
\chi_k(z)+d_k,
\end{equation}
where the constants $d_k$ are chosen so that $\Phi_k(0)=0$. Convergence of the 
integral in the right-hand side follows from the estimate \eqref{sine-type}, 
$\Phi_k$ are well-defined and  analytic functions outside $\gamma_k$. 
Furthermore, $\Phi_k$ can be extended as an entire function of exponential type 
$\pi$: it follows from the Sokhotskii-Plemelj formula that $\Phi_k$ are 
continuous on $\gamma_k$, so the singularities along $\gamma_k$ can be removed. 
The growth estimates are straightforward.  We also have 
\begin{equation}
\label{delta}
\Phi(\la_j)=0, \ j\leq k, \quad \Phi(\la_j)=1, \ j> k, 
\end{equation}
and respectively
\begin{equation}
\label{delta1} 
\Phi_k(\la_{j+1})-\Phi_k(\la_j)=\delta_{k,j}.
\end{equation}

\medskip

\noindent{ \it Formal solution to the interpolation problem:}  

Given a sequence $a=\{a(\la_k)\}\in \li(\La)$ denote
\[
\Delta_a(\la_k)=\frac{a(\la_{k+1})-a(\la_k)}{\la_{k+1}-\la_k} ,  
\]
then the function 
\begin{equation}
\label{solution}
F(z)=\sum_k  (\la_{k+1}-\la_k)\Delta_a(\la_k) \Phi_k(z).
\end{equation}  
yields a solution to  the interpolation problem  \eqref{interpolation} provided 
that the series  in the right-hand side is convergent to a function in $\Ba$.

\medskip

\noindent{ \it Solution to the interpolation problem. Convergence:} 
\begin{claim}
The series \eqref{solution} converges uniformly  on compact sets in $\C$, to an
entire function $F\in\Ba$.  This function provides a solution to the
interpolation problem
\eqref{interpolation}.
\end{claim}
\begin{proof}
It suffices to prove the convergence of the sum $\sum \Phi_k'(z)$ on each
compact
set in $\C$, to a function in $\B$. The convergence of \eqref{solution}  will
then
follow due to the normalization $\Phi_k(0)=0$. 

We remind that a set  $E\subset \R$ is called \emph{relatively dense} if, for 
some $L>0$, 
\[
\inf_{x\in \R} \text{mes} (E\cup(x,x+L)) 
\]
For example the set
\beq
\label{sete}
E=\left \{x\in \R; |x-(\lambda_k+ \frac \alpha {10} )|< \frac \alpha {20} 
\right \}
\eeq
is relatively dense, since $D^-(\La)>0$. Observe also that $\dist(E, \cup 
\gamma_k )>0$, here $\dist$ stands for the usual Euclidean distance.  

We will use the following fact (see e.g. \cite{Katsnelson73, LS74})

{\em Given a relatively dense set} $E\subset \R$ {\em there exist a constant 
$C$ such that}
\[
\sup_\R |F(x)|  \leq C \sup_E |F(x)| 
\]
{\em for each entire function $F$ of exponential type} $\pi$.

That is why it suffices to prove the  uniform convergence of the series  $\sum \Phi_k'(z)$ 
only on the set $E$ defined by \eqref{sete}.

Let $b(\la_k)=  (\la_{k+1}-\la_k)\Delta_a(\la_k)$. Since $\{\Delta_a\}\in 
l^\infty(\La)$ and also $D^-(\La)>0$ we have 
$b\in  \ell^\infty(\La)$ and 

\[
\begin{aligned}
\sum_k \Phi_k'(x)&=\sum_k\frac{b(\la_k)}{2i\pi} \ \frac
{S_{\la_k}'(x)}{x-(\la_k+i)}\int_{\gamma_k}
           \frac{\zeta-(\la_k+i)}{S_{\la_k}(\zeta)}\frac{d\zeta}{\zeta-x}- \\
&\sum_k\frac{b(\la_k)}{2i\pi} \ \frac{S_{\la_k}(x)}{(x-(\la_k+i))^2}
\int_{\gamma_k}
           \frac{\zeta-(\la_k+i)}{S_{\la_k}(\zeta)}\frac{d\zeta}{\zeta-x}+ \\ 
&\sum_k \frac{b(\la_k)}{2i\pi} \ \frac{S_{\la_k}(x)}{x-(\la_k+i)}
\int_{\gamma_k}
           \frac{\zeta-(\la_k+i)}{S_{\la_k}(\zeta)}\frac{d\zeta}{(\zeta-x)^2}
=\\
&=\Sigma_1(x) +\Sigma_2(x)+\Sigma_3(x).
\end{aligned}
\]
We consider just the first sum, the rest  can be treated similarly. It follows
from the construction of the sine-type functions $S_{\la_k}$ that for some $C>0$
and all  $k$
\[
|S_{\la_k}(x)|<C, \ x\in \R, \qquad
\left|\frac{\zeta-(\la_k+i)}{S_{\la_k}(\zeta)}\right|<Ce^{-\pi |\Im \zeta|/2}, \
\zeta \in 
 \gamma_k,
\]
so the Cauchy inequality   gives
 \[
 \begin{split}
\left |  \frac {S_{\la_k}'(x)}{x-(\la_k+i)}\int_{\gamma_k}
           \frac{\zeta-(\la_k+i)}{S_{\la_k}(\zeta)}\frac{d\zeta}{\zeta-x} \right
| 
      \leq\\ 
      C_1 |x-(\lambda_k+i)|^{-1} |x-(\lambda_k+\alpha/10)|^{-1/2}.
 \end{split}
\]
Now the proof of the claim is straightforward.
\end{proof}
This claim completes the proof of Lemma~\ref{lemma1}.
\end{proof}

\medskip

\noindent{ \it Splitting of the sequence:}  In order to complete the sufficiency
proof  we split the sequence 
into two parts: $\La=\Gamma_1\cup \Gamma_2$ so that 
\[
\Gamma_i \ \text{are separated,} \  D^-(\Gamma_i)>0, \  D^+(\Gamma_i) <
1/2.
\]    
We remind that $\La $ already admits the  representation $\La=\La_1\cup\La_2$ 
where $\La_i$ are  separated sequences, not necessarily satisfying the density 
restrictions. Our goal is  to rearrange this splitting. We enumerate the 
sequence $\Lambda$ in an increasing order, i.e. $\Lambda=\{\lambda_k\}_{k\in 
\Z}$ with $\lambda_k<\lambda_{k+1}$. We define 
$\Gamma_1=\{\lambda_{2k}\}_{k\in\Z}$ and $\Gamma_2= 
\{\lambda_{2k+1}\}_{k\in\Z}$. This splitting satisfy the desired properties.

\medskip

\noindent{ \it Now we complete the proof of  the general case.}   We use a trick 
from \cite{BoeNic04}.  Consider the splitting $\La=\Gamma_1\cup \Gamma_2$ as 
above. It follows from Lemma \ref{lemma1} that each sequence $\Gamma_i$ is 
interpolating in $B^1_{\pi/2}$. Since $D^+(\Gamma_1)<1/2$ we can construct a 
sine-type function $S$ of type $\pi/2$ vanishing on $\Gamma_1$ (and, perhaps, at 
some other points).   Given a sequence $a\in \li$ we observe that the sequences 
$a|_{\Gamma_i}$ belong to the spaces $l^\infty_1(\Gamma_i)$, $i=1,2$. We look 
for the solution of the  problem \eqref{interpolation} in the form
\begin{equation}
\label{solution1}
F(z)=H_1(z)+S(z) H_2(z),  \  H_1, H_2 \in B^1_{\pi/2}.
\end{equation}
We observe that $SH\in \Ba$ if $H \in B^1_{\pi/2}$. Let  $H_1\in B^1_{\pi/2}$
solve the interpolation problem
\[
      H_1|_{\Gamma_1}= a|_{\Gamma_1},
\] 
then $H_2$ should satisfy 
\begin{equation}
\label{h2}
H_2(\mu)=\frac{a(\mu)- H_1(\mu)}{S(\mu)}, \ \mu\in \Gamma_2.
\end{equation}
We will prove that the right-hand side of \eqref{h2} is bounded, then, since
$\Gamma_2$ is separated the equation has a solution in $B_{\pi/2}$. The
boundedness is straightforward: let $\mu'$  be the nearest  point to $\mu$ 
 in $ \Gamma_1$.  We then have
\[
\frac{a(\mu)- H_1(\mu)}{S(\mu)}= \frac{a(\mu)- a(\mu')}{S(\mu)}- \frac{H_1(\mu)-
H_1(\mu')}{S(\mu)}
\]
It follows from the definition of the sine-type function that  there are  $\eps 
>0$, and $c>0$ such that $|S(x)-S(\mu')|  \simeq |x-\mu'|$ if $\mu'\in 
\Gamma_1$,  $|x-\mu'|<\eps$,  and  that$|S(x)|>c$ if dist$(x, \Gamma_1)> \eps$.  
 Since $|\mu-\mu'|<R$, $a\in \li(\La)$, and $H_1\in B^1_{\pi/2}$ the boundedness 
of the right-hand side in \eqref{h2} follows.


\section{Traces}\label{sect:traces}

Let $S$ be  a sine-type function  such  that its zero set 
$\Lambda=\{\lambda_n\}_{-\infty}^\infty \subset \R$, for simplicity  assume that 
$0\notin \Lambda$. In this section we study  the traces of functions in $\Ba$ on 
$\La$. In the classical case of the space $\B$ and $\La=\Z$   the traces can be 
described in terms of boundedness of the corresponding discrete Hilbert 
transform see e.g. \cite{Levin56}, Appendix VI,  \cite{Levin96}, Lecture 21 
when $S(z)=\sin \pi z$. We refer the reader to \cite{Eoff, Flornes} for other 
spaces of entire functions of exponential type. In the case of the space $\Ba$ 
one needs   in addition a regularization of the Hilbert transform. 

We introduce some additional notation. 
Given $x\in \R$ we denote by $\dl x \dr=\max \{n;\la_n<x\}$ and, for a sequence  
 $\ba=\{a_n\}$ we define the usual  and regularized  Hilbert transforms (with 
respect to $S(\la)$) at any point $x\notin \Lambda$)   as 
\begin{equation}
\label{hilberttransform}
(\cH \ba)(x) = \lim_{N\to \infty} \sum_{n=-N}^N \frac{a_n}{S'(\la_n)} \left (
           \frac 1 {x - \lambda_n} + \frac 1 {\lambda_n}  \right ),
\end{equation}
and
\begin{equation}
\label{reghilberttransform}
(\tilde{\cH} \ba)(x) = \lim_{N\to \infty} \sum_{n=-N}^N \left [ 
\frac{a_n}{S'(\la_n)}   
 \left (
           \frac 1 {x - \lambda_n} + \frac 1 {\lambda_n}  \right ) -
\frac{a_{\dl x\dr}}{S(x)} \right ]
\end{equation}
assuming that the limits exist. The additional term in the right-hand side of 
\eqref{reghilberttransform} regularizes the behavior at $\infty$: the sequence 
$\ba$ needs not to be bounded, we will consider the cases  when $  (\tilde{\cH} 
\ba)(x)$ is bounded for large values of $x$, in contrast to $(\cH \ba)(x)$.

\begin{theorem}
\label{thm:traces}
Let $B\subset \R$ be any separated  sequence  such that
\[
\dist(B,\La )> 0, \quad D^-(B)>1.
\]
Given  a sequence $a\in \li(\La)$ there exists a function  $f\in \Ba$ such that 
\begin{equation}
\label{interpolation1}
a_k=f(\la_k)
\end{equation}
if and only if  
\begin{equation}
\label{inequality}
\{\tilde{\cH} \ba(\beta)\}_{\beta\in B} \in \ell^\infty(B).
\end{equation}
 
\end{theorem}

\noindent{\bf Remark.} We will see that this relation is independent of the  
choice of the sequence $B$.\\
\noindent{\it Proof.} Let a sequence $a\in \li(\La)$ satisfy \eqref{inequality}.
We replace  the problem \eqref{interpolation1}  by 
\begin{equation}
\label{interpolation2}
f(\la_{k+1}) - f(\la_k) = a_{k+1}-a_k=:b_k, \  f\in \Ba,
\end{equation}
and look for the solution of this problem in the form 
\begin{equation}
\label{interpolatingseries}
f(z)=\sum_nb_n\Psi_n(z),
\end{equation} 
where the functions $\Psi_n\in \Ba$ satisfy the equations
\begin{equation}
\label{deltaint}
\Psi_n(\la_{k+1})- \Psi_n(\la_k)= \delta_{k,n}
\end{equation}
and, in addition,
\begin{equation}
\label{normalization}
\Psi_n(0)=0.
\end{equation}

\medskip

The idea of constructing such functions is the same as in the previous section, 
yet its realization is slightly different. Take $\alpha_n \in 
(\lambda_n,\lambda_{n+1})$ so that $3\kappa:=\dist(\{\alpha_n\}, \Lambda\cup 
B)>0$ and such that  $(\la_n,\alpha_n)\cap B=\emptyset$ . Denote 
$\gamma_n=\alpha_n+i\R$ and let
 \[
\chi_n(z)=\begin{cases}
                   1,  & \Re z >  \alpha_n; \\
                   0, & \mbox{otherwise}.
             \end{cases}
\]
The functions
\begin{equation}
\label{elementaryblock} \Psi_n(z)= \frac{S(z)}{2i\pi}\int_{\gamma_n} \left (
\frac 1{\zeta-z} - \frac 1 \zeta \right ) \frac{d\zeta}{S(\zeta)}  + \chi_n(z)
+d_n,                                   
\end{equation}
 belong to $\Ba$ and satisfy \eqref{deltaint}, \eqref{normalization} for an
appropriate choice of $d_n$'s.
 
\medskip

\begin{claim}
Given any sequence $\bb=\{b_n\}\in \ell^\infty$ the series
\eqref{interpolatingseries} converges uniformly on compact sets in $\C$ to an
entire function  $f$ of exponential type $\pi$. 
\end{claim}
\begin{proof}
We will prove the convergence of the series  
\begin{equation}
\label{sum}
\sum_nb_n\Psi_n'
\end{equation}
 and then  use  \eqref{normalization}. We write 
\begin{equation}
\label{derivative} \Psi_n'(z)= \underbrace{ \frac{S'(z)}{2i\pi} \int_{\gamma_n}
\left ( \frac 1 {z-\zeta} + \frac 1 \zeta \right ) \frac{d\zeta}{S(\zeta)}
}_{I_n(z)} + \underbrace{ \frac{S(z)}{2i\pi} \int_{\gamma_n} \frac 1
{(z-\zeta)^2}     \frac{d\zeta}{S(\zeta)}      }_{J_n(z)}                       
\end{equation}
and consider the series
\begin{equation}
\label{twoseries}
\sum b_n I_n(x) \quad \mbox{and} \quad \sum b_n J_n(x).
\end{equation}
separately. Given a compact set $K\subset \C$ we   prove the convergence in the 
set
$\{z\in K; \ \dist (K\cup (\cup \gamma_n ) ) > \kappa/2\}$, for the remaining
piece of $K$ we replace
the lines $\gamma_n$ in \eqref{elementaryblock} by $\gamma_n'=\gamma_n+\kappa$.
This   does not change $\Psi_n$, and we repeat the same reasonings. 

Now the uniform convergence on $K$ follows from  the estimates
 \begin{equation}
\label{auxiliary1}
  | z-\zeta  |  \gtrsim \dist(\gamma_n, K),  \ z\in K, \zeta \in \gamma_n,
  \end{equation}
  and
  \begin{equation}
  \label{auxiliary2}
  |S(\zeta)| \gtrsim e^{\pi |\Im \zeta|}, \ \zeta \in \cup_n \gamma_n.
\end{equation}
\end{proof}

The function $f$ defined by \eqref{interpolatingseries} may not always belong 
to $\Ba$.

\begin{claim} Let $\ba\in \ell^\infty_1$ satisfy \eqref{inequality},
$\bb=\{b_n\}$, where
$b_n=a_{n+1}-a_n$ and $f$ is defined by \eqref{interpolatingseries}. Then $f\in
\Ba$.
\end{claim}
\begin{proof}
It is straightforward that, for any $\bb \in \ell^\infty$, the sum $ \sum b_n 
J_n(x)$ in \eqref{twoseries}   is uniformly bounded on $B$. So it suffices to 
prove that the first sum $ \sum b_n I_n(x)$ is also bounded on $B$. Then the 
function $f'$ itself will be also bounded on $B$, so one can once again refer to 
\cite{Beurling89a} to get boundedness everywhere.

\medskip

In order to estimate the first sum in \eqref{twoseries}  we observe that one can 
apply the residue theorem in the halfplane $\Re \zeta > \alpha_n$ in order to 
express  $I_n(x)$:

\begin{equation}
\label{ian}
I_n(x)=  S'(x) \sum_{j>n}\left (
       \frac 1{x-\lambda_j} + \frac 1 {\lambda_j}     \right  )\frac 1
{S'(\lambda_j)}
                +\frac{\chi_n(x)}{S(x)},
\end{equation}   
                              
Respectively  
 
\begin{multline}
\label{wholesum}
A(x):=  
             \sum_{-\infty }^\infty b_nI_n(x)=
             \\
 \lim_{N\to \infty }  \left (
\underbrace{
   \sum_{-N}^N b_n  
            \sum_{j>n}                              \left (
   \frac 1{x-\lambda_j}+ \frac 1 {\lambda_j}
                                                   \right )    \frac 1
{S'(\lambda_j)}
                       }_{A_1(x)} +
 \underbrace {
\frac 1{S(x)} 
        \sum_{-N}^N b_n  \chi_n(x)
                   }_{A_2(x)}
                   \right )
\end{multline}

Since $b_n=a_{n+1}-a_n$ we have

\begin{multline}
\label{firstsum}
A_1(x)= \sum_{-N+1}^N a_n 
                                         \left (
   \frac 1{x-\lambda_n}+ \frac 1 {\lambda_n}
                                                   \right )    \frac 1
{S'(\lambda_n)} -  \\
        a_{-N}     \sum_{j>-N}                              \left (
   \frac 1{x-\lambda_j}+ \frac 1 {\lambda_j}
                                                   \right )    \frac 1
{S'(\lambda_j)} +  \\                                                
       a_{N+1}     \sum_{j>N}                              \left (
   \frac 1{x-\lambda_j}+ \frac 1 {\lambda_j}
                                                   \right )    \frac 1
{S'(\lambda_j)} = B_1(x)+B_2(x)+B_3(x).
          \end{multline}

For large $N$ we have $\chi_N(x)=0$ therefore, using \eqref{ian} once again and 
then applying the dominated convergence theorem we obtain 
\[
B_3(x)= S'(x)^{-1} a_N I_N(x) \to 0 \ \ \mbox{as} \ \ N\to \infty
\]
Besides, since
\[
a_{-N} \left (
\frac 1 {x-\lambda_{-N}} + \frac 1 {\lambda_{-N}}
            \right )  = o(1), \ \mbox{as}  \ N\to \infty,
\]
we may assume that the summation in $B_1(x)$ is taken from $-N$ to $N$.     

Similarly we have 
\begin{multline}
\label{secondsum}
A_2(x)= \\ \left [ 
  \sum_{-N+1}^N a_n(\chi_{n-1}(x)-\chi_n(x))   -a_{-N}\chi_{-N}(x) +
             a_{N+1} \chi_{N+1}(x)
              \right ]  \frac 1 {S(x)} .
\end{multline}           

Recall that we denote $k=\dr x \dl \in \Z$ so that $\gamma_k < x < 
\gamma_{k+1}$. Then the only
non-zero summand in the first term in 
\eqref{secondsum} is $a_{\dr x \dl}$. Besides $\chi_{N+1}(x)=0$ for sufficiently
big $N$.

We substitute  this  together with   \eqref{firstsum}   in \eqref{wholesum}  and
observe 
(using \eqref{ian} once again) that 
\[
B_2(x)+a_{-N}\chi_{-N}(x) =I_{-N}(x) \to 0 \ \mbox{as} \ N\to \infty.
\]

Finally we have 
\begin{multline}
\label{finally}
A(x)= \lim_{N\to \infty}( B_1(x)- \frac {a_{\dr x \dl}}{S(x)} )= \\
 \lim_{N\to \infty} \left \{ 
 \sum_{-N}^N a_n 
                                         \left (
   \frac 1{x-\lambda_n}+ \frac 1 {\lambda_n}
                                                   \right )    \frac 1
{S'(\lambda_n)}  - \frac {a_{\dr x \dl}}{S(x)}
                          \right \}
\end{multline}

Boundedness  of this expression in $x \in B$ is equivalent to the boundedness of 
$(\tilde{\cH} \ba)_{\dr x \dl}$. Thus, the function $f$  defined by 
\eqref{interpolatingseries} belongs to $\Ba$ and solves the interpolation  
problem \eqref{interpolation} modulo an additive constant. This proves the 
``if'' part of Theorem \ref{thm:traces}.

\medskip

 In order to prove the ``only if'' part of Theorem \ref{thm:traces} we need an
auxiliary statement.
\begin{lemma}\label{decomposition} Given an $\eps >0$, each  function $f\in \Ba$
admits the
representation:
\begin{equation}
\label{expansion}
f=f_1+f_2,
\end{equation}
where $f_1\in \Ba$ is of exponential type at most $\eps$ and $f_2\in \B$.   
\end{lemma}
We postpone the proof of this lemma until the end of the section.
 
\medskip

Now, given $f\in \Ba$ denote $\ba= \{f(\la_n)\}$. In order to prove
\eqref{inequality} we chose $\eps=\pi/2$ and we use the representation
\eqref{expansion}. Then 
\[
\ba=\ba^{1}+\ba^{2},  \quad  \ba^i=\{f_i(\la_n)\}. 
\]
That $\tilde{\cH} \ba^2 \in \ell^\infty(B)$  is straightforward since $\ba^2$ is
the trace on $B$ of a function from $\B$ and we can use the known results from
\cite{Levin96}.

\medskip 

In order to study the Hilbert transform of $\ba^1$ we consider the integral 
\[
I_R(x)=\frac 1 {2i\pi} \int_{\Gamma_R} \frac{f_1(\zeta)}{S(\zeta)}
       \left ( \frac 1 {x-\zeta} +\frac 1 \zeta\right ) d\zeta,
\]
where $\Gamma_R=\{\zeta; |\zeta|=R\}$. We consider only those values of $R$ for
which 
$\dist(\Gamma_R, \La) > \alpha$ for some fixed  $\alpha >0$.   

Let $\zeta=\xi+ i \eta$. We have $|f_1(\zeta)|=O(|\zeta|e^{\pi |\eta|/2})$,
$\zeta \in \C$, 
$|S(\zeta)| \gtrsim \  e^{\pi |\eta| }$, $\zeta \in \Gamma_R$ and, by the
Jordan lemma,
\[
I_R(x) \to 0, \ {\rm as}  \   R\to \infty. 
\]
On the other hand  the residue theorem gives
\[
I_R(x)=\sum_{|\la_n|<R}  \sum_{-N}^N f_1(\la_n) 
                                         \left (
   \frac 1{x-\lambda_n}+ \frac 1 {\lambda_n}
                                                   \right )    \frac 1
{S'(\lambda_n)}  - \frac {f(x)}{S(x)}
                          + \frac{f_1(0)}{S(0)}.
\]
We obtain 
\[
\tilde{\cH} \ba^1(x)= - \frac{f_1(0)}{S(0)}
\]
for all $x\in \R$. This  of course  suffices to complete the proof of the
theorem.  
\end{proof}

We proceed now with the proof of Lemma~\ref{decomposition}. We say that a
function
$f$ has a spectral gap at the origin if there is an $\epsilon>0$ such that 
$\supp \hat f \cap (-\varepsilon,\varepsilon)=\emptyset$.
We start with the
following Lemma:
\begin{lemma}\label{gap}
 Let $f\in B^1_\pi$ have a spectral gap at the origin, then  $f\in B_\pi$.
\end{lemma}
\begin{proof}
We start by observing that if $f\in B_\pi^1$ then $g=f'$ is bounded on $\R$
as observed in the Introduction.  Since $\hat g(\omega) = 2\pi i\omega \hat
f(\omega)$, $g$ has also a spectral gap.
Let $\phi$ be a compactly supported smooth   function such that
$\phi(w)=1/(2\pi i w)$ when $\varepsilon<|w|<\pi$. Thus $\hat f(\omega)
=\phi(\omega) \hat g(w)$. Therefore $f=-\hat \phi \star g$. Since $-\hat \phi$
belongs to the Schwartz space and $g$ is bounded then $f$ is itself bounded.
\end{proof}
\begin{proof}[Proof  of Lemma~\ref{decomposition}]
Given $\varepsilon>0$ it is possible to find, with a partition of unity, two
smooth, compactly supported functions $\phi$ and $\psi$ such that
$\supp \phi\subset [-2\varepsilon,2\varepsilon]$, $\phi(\omega)=1$ if
$|\omega|<\varepsilon$ and $\phi(\omega)+\psi(\omega)=1$
for all $\omega\in [-\pi,\pi]$. Let $\Phi,\Psi$ be the two functions in the
Schwartz space such that $\hat \Phi=\phi$ and $\hat \Psi=\psi$. Now given $f\in
B_\pi^1$ we can decompose 
\[
f=\Phi\star f + \Psi\star f=f_1+f_2.
\]
The function $f_1'=\Phi\star f'$, thus it is bounded on the real line. Moreover
its spectrum lies in $[-2\varepsilon,2\varepsilon]$ because $\hat f_1 = \phi\hat
f$. Finally $f_2'=\Psi\star f'$, and therefore it is bounded on the real line.
The spectrum of $f_2$ is contained in the spectrum of $f$, thus $f_2\in
B_\pi^1$. Moreover $\hat f_2 = \psi \hat f$ and $\psi(\omega)=0$ when
$|\omega|<\varepsilon$, thus $f_2$ has a spectral gap at the origin. By
Lemma~\ref{gap}, $f_2$ is bounded.
\end{proof}

We finish by observing some elementary remarks on the zero sets of functions in
$B^1_\pi$. When the functions are in the Bernstein class, its zeros have been
studied, \cite{Khabibullin11}. The novel case is when $f\in B^1_\pi\setminus
B_\pi$. In this setting we can prove

\begin{claim}
 \label{density}
 Given any $f\in B^1_\pi \setminus B_\pi$ and $A>0$ let $Z_A(f)$ be the set of
 zeros of $f$ located in the strip
 $ \{z\in \C; |\Im z|<A\}$. Then $D^-(Z_A(f))=0$.
 \end{claim}

{\bf Remark.} Here we use a slightly modified notion of density  which counts 
the number of points in a strip rather than on the real line.
\begin{proof}
Suppose that $|f(x)|=M$. Since $f$ is Lipschitz there is an interval $I$
centered at $x$ and of size comparable to $M$ such that $f$ is zero free in
$I$. Thus if $|f(x_n)|\to\infty$ there are arbitrary big gaps $I_n$ without
zeros. Therefore $D^-(Z_A(f))=0$.
\end{proof}
On the other hand given any $\varepsilon>0$, since $\Lambda=(1-\varepsilon)\Z$
is an interpolating sequence it is possible to construct a  function
$f\in B^1_\pi\setminus B_\pi$ vanishing on most points of $\Lambda$ and
such that 
\[
 \limsup_{R\to\infty} \frac {\#\{Z(f)\cap (-R,R)\}}{2R}>(1-\varepsilon). 
\]
This can be done prescribing the value $0$ in very long intervals of $\Lambda$ 
alternating with shorter intervals in $\Lambda$ (but still of increasing lenth) 
where the values are bigger and bigger.


\begin{thebibliography}{Tha11}




\bibitem[B89]{Beurling89a}
A. Beurling, \emph{The collected works of {A}rne {B}eurling. {V}ol. 2},
  Contemporary Mathematicians, Birkh\"auser Boston Inc., Boston, MA, 1989.
  
\bibitem[BN04]{BoeNic04}
B. B{\o}e and A. Nicolau, \emph{Interpolation by functions in the
  {B}loch space}, J. Anal. Math. \textbf{94} (2004), 171--194.  

 \bibitem[E95]{Eoff} C.Eoff, \emph{The discrete nature of the Paley-Wiener 
spaces}, Proc. Amer. Math. Soc. 123 (1995), no. 2, 505–512. 

\bibitem[F98] {Flornes} K.Flornes, {\em Sampling and interpolation in the 
Paley-Wiener spaces $L^p_\pi$, $0<p≤1$}, Publ. Mat. 42 (1998), no. 1, 103–118.

\bibitem[G07]{Garnett} J. Garnett,  \emph{ Bounded analytic functions},   
Springer, New York, 2007. xiv+459 pp.  

\bibitem[K11]{Khabibullin11}B. Khabibullin \emph{Distribution of zero
subsequences for Bernstein space and criteria of completeness for exponential
system on a segment}, arXiv:1104.2683

\bibitem[Ka73] {Katsnelson73} V.Kacnelson,  
\emph{Equivalent norms in spaces of entire functions}, (Russian)
Mat. Sb. (N.S.) 92(134) (1973), 34–54. (English translation: Math. USSR-Sb. 21 
(1973), 33–55).

\bibitem[L56]{Levin56} B. Levin, 
\emph{Distribution of zeros of entire functions} (Russian) Gosudarstv. Izdat. 
Tehn.-Teor. Lit., Moscow, 1956, 632 pp. English translation:  American 
Mathematical Society, Providence, R.I. 1964.

\bibitem[L96]{Levin96}
B. Levin, \emph{Lectures on entire functions}, Translations of
Mathematical
  Monographs, vol. 150, American Mathematical Society, Providence, RI, 1996.

\bibitem[LS74]{LS74} V.Logvinenko, Ju. Sereda,
\emph{Equivalent norms in spaces of entire functions of exponential type}, 
(Russian) Teor. Funkciĭ Funkcional. Anal. i Priložen. Vyp. 20 (1974), 102–111.

\bibitem[LM05] {LM05} Yu. Lyubarskii and W. Madych,  
\emph{Interpolation of functions from generalized Paley-Wiener spaces}, 
J. Approx. Theory 133 (2005), no. 2,  251 - 268.  
  
\bibitem[OCS98]{OrtSei98}
J. Ortega-Cerd{\`a} and K. Seip, \emph{Beurling-type density
  theorems for weighted {$L\sp p$} spaces of entire functions}, J. Anal. Math.
  \textbf{75} (1998), 247--266.  

\bibitem[T11]{Thakur11}
G. Thakur, \emph{Bounded mean oscillation and bandlimited interpolation in
  the presence of noise}, J. Funct. Anal. \textbf{260} (2011), no.~8,
  2283--2299.  

\end{thebibliography}

\def\cprime{$'$}
\providecommand{\bysame}{\leavevmode\hbox to3em{\hrulefill}\thinspace}
\providecommand{\MR}{\relax\ifhmode\unskip\space\fi MR }
\providecommand{\MRhref}[2]{%
  \href{http://www.ams.org/mathscinet-getitem?mr=#1}{#2}
}
\providecommand{\href}[2]{#2}

\end{document}